\documentclass[12pt]{amsart}
\usepackage[T1]{fontenc}
\usepackage{mathptmx}

\usepackage{comment}
\usepackage[colorlinks=true,linkcolor=magenta,citecolor=blue]{hyperref}     
\usepackage{mathtools}
\usepackage{amsmath, amssymb,amsthm}
\usepackage[capitalise]{cleveref}   
\usepackage{array}
\usepackage{graphicx}
\usepackage{float}
\usepackage{caption,subcaption}
\usepackage{tikz}
\usepackage{tikz-cd}
\usetikzlibrary{matrix,arrows}
\usepackage[mathscr]{euscript}
\usepackage{xcolor}
\usetikzlibrary{arrows.meta,calc,positioning}

\def\Cbb{\mathbb{C}} 
 
\def\Gbb{\mathbb{G}}

 \def\Pbb{\mathbb{P}}
\def\Qbb{\mathbb{Q}}

 \def\Zbb{\mathbb{Z}}

\def\Ecal{\mathcal{E}}

\newtheorem{thm}{Theorem}[section]
\newtheorem{lem}[thm]{Lemma}

\newtheorem{defi}[thm]{Definition}
\newtheorem{cor}[thm]{Corollary}

\begin{document}
\author{Yugang Zhang}
\address{Yugang Zhang, Université Bourgogne Europe, CNRS, IMB UMR 5584, 21000 Dijon, France}
\email{yugang.zhang@ube.fr}
\urladdr{https://sites.google.com/view/yugangzhang}

\thanks{The author acknowledges support by the programme ATRACT of Région Bourgogne-Franche-Comté under the project ADYAUS (RECH-ATRAC-000012) and the French National Research Agency under the project DynAtrois (ANR-24-CE40-1163)}

\title{Special regular polynomial skew products}
\maketitle

\begin{abstract}
We define a regular polynomial skew product $(p(z),q(z,w))$ of $\mathbb{C}^2$ of degree $d\geq 2$ to be special if it is triangularly conjugate to a map of the form $(p(z),q(w))$, where $p$ and $q$ are power maps or $\pm$Chebyshev maps, or of the form $(z^d,D_d(w,\zeta z^m))$, where $\zeta^{d-1}=1$, $m\in\{1,2\}$, and $D_d$ is the Dickson polynomial of degree $d$. We justify this definition by showing the following equivalence.
\begin{enumerate}
    \item $f$ is special.
    \item $f$ is semiconjugate to an affine self-map $g$ in skew product form of a 2-dimensional connected and commutative algebraic group $G$ over $\mathbb{C}$.
    \item All multipliers of $f$ are contained in a fixed number field $K$.
\end{enumerate}
This generalizes the one-variable polynomial case.
\end{abstract}

\section{Introduction}
\subsection{Special rational maps}
Following the standard classification, a rational map of a Riemann sphere special is usually called \emph{special} if it is semiconjugate to an affine map of a connected commutative algebraic group $G$. More precisely, if $G=\Gbb_m$ is an algebraic torus, then $f$ is conjugate to the power map $z\mapsto z^{\pm d}$ or to $\pm T_d$ (where $T_d$ is the \emph{Chebyshev polynomial} of degree $d$, characterized by $T_d\left(z+\frac{1}{z}\right)=z^d+\frac{1}{z^d}$). If $G$ is an elliptic curve, then $f$ is a \emph{Latt\`es} map by definition. A Latt\`es map is called \emph{flexible} if the associated affine map of the elliptic curve is multiplication by an integer; see Lemma~\ref{lem:onedim-affine-semi}.

Special maps play an important role in complex dynamics since many rigidity phenomena hold only either for special maps or for non special maps. For example, if we denote by $\mu_f$ the unique maximal entropy invariant measure of $f$ and $\alpha$ its Hausdorff dimension, then Zdunik's rigidity theorem~\cite{Zdunik1990} states that $\mu_f$ is absolutely continuous with respect to the $\alpha$-Hausdorff measure on the Julia set if and only if $f$ is special. 

Another rigidity property that we would like to generalize in this paper concerns the field of definition of the multipliers of $f.$ A point $z\in \Pbb^1_\Cbb$ is \emph{periodic of exact period} $n\geq 1$ if $f^n(z)=z$ and $n$ is the smallest positive integer with this property, where $f^n=f\circ\cdots\circ f$ denotes the $n$-fold iterate of $f$. The multiplier of $z$ is the eigenvalue of the differential $Df^n(z)$ of $f^n$ at $z$. It describes the first-order asymptotic behavior of the dynamical system near $z$ and plays an important role in complex dynamics. McMullen~\cite[Corollary 2.3]{mcmullen87} proved that, aside from flexible Latt\`es maps, the set of all multipliers determines the conjugacy class of $f$ up to finitely many possibilities. Ji--Xie~\cite[Theorem 1.3]{JiXie2025MultiplierSpectrum} later generalized this result and showed that, generically in the Zariski topology, it in fact determines the conjugacy class of $f$ uniquely. See also Huguin~\cite[Theorem C]{huguin2024modulispacespolynomialmaps} and Ji--Xie~\cite[Theorem 1.4]{JiXie2025MultiplierSpectrum} for the polynomial case.

A special rational map has all its multipliers contained in the ring of integers of an imaginary quadratic field (see~\cite{milnor06}). Milnor~\cite{milnor06} conjectured that the converse also holds, and this was first proved by Ji--Xie~\cite[Theorem 1.13]{JiXiePi}. See also Huguin~\cite[Theorem 9]{Huguinquadratic} for the case of quadratic rational maps, and also Buff--Gauthier--Huguin--Raissy~\cite[Theorem 2]{arXiv:2212.03661} for entire maps of $\Cbb$.

Huguin then proved a stronger rigidity result that we will use later.
\begin{thm}[Huguin {\cite[Theorem 7]{Huguinpolytechnique}}]~\label{thm: Huguin}
    A rational map whose multipliers are all contained in a fixed number field is special.
\end{thm}
Ji-Xie-Zhang~\cite[Theorem 1.5]{JXZAnn26} strengthened Theorem~\ref{thm: Huguin} by requiring only that, for any multiplier $\lambda$, some power of $|\lambda|$ lies in a fixed number field.

\subsection{Special regular polynomial skew products}
A \emph{(regular polynomial) skew product} of degree $d\geq 2$ is a polynomial map $f$ of $\Cbb^2$ of the form
\begin{align}
    f(z,w)=(p(z),q(z,w))
\end{align}
that extends holomorphically to all of $\Pbb^2_\Cbb$, where $p$ and $q$ are polynomials of degree $d$. We often write $q_z(w)$ for $q(z,w)$. This class of maps was studied in detail by Jonsson~\cite{Jonsson} and exhibits interesting features that do not occur for rational maps or for general polynomial endomorphisms of $\Pbb^2_\Cbb$, see e.g.,~\cite{Wandering16,Dujardin16,Dujardin17,Taflin21,ab1,ab2,tapiero}.

For a point $(z,w)\in \Cbb^2$, the iterate $f^n(z,w)$ has the form $f^n(z,w)=\bigl(p^n(z),Q_z^n(w)\bigr),$
where $Q_z^n(w):=q_{p^{n-1}(z)}\circ\cdots\circ q_z(w).$
Therefore the differential of $f^n$ at $(z,w)$ is
\begin{align*}
D f^n(z,w)=
\begin{pmatrix}
(p^n)'(z) & 0\\
\partial_z Q_z^n(w) & \partial_w Q_z^n(w)
\end{pmatrix}.
\end{align*}
In particular, if $(z,w)$ is a periodic point of period $n$, then its \emph{multipliers}, which by definition are the eigenvalues of 
$D f^n(z,w)$, are $(p^n)'(z)$ and $\partial_w Q_z^n(w).$ We say that a skew product is \emph{triangularly conjugate} to another skew product if the conjugation is an affine automorphism of the form $(a(z),b(z,w))$.

To state our main result, we recall the definition of a Dickson polynomial.
It was introduced by Dickson~\cite{Dickson} and has been studied mostly over finite fields; see the book~\cite{bookdickson}. The \emph{Dickson polynomial of degree} $d$ is the unique polynomial $D_d(X,A)\in \Zbb[X,A]$ such that
\begin{align*}
D_d\left(t+\frac{a}{t},a\right)=t^d+\frac{a^d}{t^d}.
\end{align*}
It generalizes both power maps and Chebyshev maps. Indeed, $D_d(t,0)=t^d$ and $D_d(t,1)=T_d.$
One can show that, for example, $D_2(x,a)=x^2-2a$ and $D_3(x,a)=x^3-3ax$.

Set
\begin{align}
    \Ecal_d := \{z\mapsto z^d , \pm T_d\}.
\end{align}
Then a one-variable polynomial $p$ is special if and only if $p$ is conjugate to a map in $\Ecal_d$.

We now state our definition of special skew product maps.
\begin{defi}\normalfont
    A skew product is \emph{special} if it is triangularly conjugate to one of the following maps
\begin{align}
&(P(z),Q(w)), \quad \text{where } P,Q\in\Ecal_d
\tag{$\dagger1$}\label{eq:tagger1}.\\
&f_{2,m}(z,w)=\Bigl(z^d,D_d\bigl(w,\zeta z^m\bigr)\Bigr),
\quad \text{where } \zeta^{d-1}=1 \text{and } m\in\{1,2\}
\tag{$\dagger2$}\label{eq:tagger2}.
\end{align}
\end{defi}
The case ~\eqref{eq:tagger2} is ``nontrivial''. For instance, if $p(z)=z^d$ and $\zeta=1$, then over the fixed fiber $z=0$ we obtain the power map $D_d(w,0)=w^d,$
whereas over the fixed fiber $z=1$ we obtain the Chebyshev map $D_d(w,1)=T_d.$ 

In particular, \eqref{eq:tagger2} is not conjugate to~\eqref{eq:tagger1}. 
Nevertheless, one can remark that \eqref{eq:tagger2} is semiconjugate to $\big(z^d,\pm T_d(w) \big)$. In fact, if we set $\Pi_m(u,v)=(u^2, u^mv)$, then $f_{2,m}\circ\Pi_m=\Pi_m \circ g$, where $g(u,v)=\big(u^d, D_d(v,\zeta) \big)$ is affinely conjugate to $\big(u^d,\pm T_d(v) \big)$ by Lemma~\ref{lem:dickson-conjugacy}.

\subsection{Semiconjugacy}
In this text, algebraic groups are connected and commutative by definition. The following digram would be helpful to the reader to understand the following definitions. In particular, Definition~\ref{def: semiconjugate} means that the diagram is commutative.
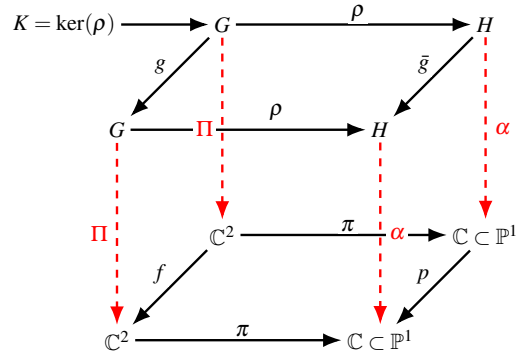
\begin{figure}[ht]
\centering
\begin{tikzpicture}[
    scale=0.7,
    transform shape,
    >=Latex,
    line width=0.9pt,
    font=\normalsize,
    map/.style={->, draw=black},
    vmap/.style={->, draw=red, dashed},
    obj/.style={
        text=black,
        fill=white,
        inner sep=2pt,
        outer sep=1pt
    },
    lab/.style={
        fill=white,
        text=black,
        inner sep=1.5pt
    },
    vlab/.style={
        fill=white,
        text=red,
        inner sep=1.5pt
    }
]

\node[obj] (K)  at (-3,4) {$K=\ker(\rho)$};
\node[obj] (G1) at (0,4) {$G$};
\node[obj] (H1) at (5,4) {$H$};

\node[obj] (G2) at (-2,2) {$G$};
\node[obj] (H2) at (3,2) {$H$};

\node[obj] (C1) at (0,0) {$\mathbb{C}^2$};
\node[obj] (P1) at (5,0) {$\mathbb{C}\subset \mathbb{P}^1$};

\node[obj] (C2) at (-2,-2) {$\mathbb{C}^2$};
\node[obj] (P2) at (3,-2) {$\mathbb{C}\subset \mathbb{P}^1$};

\draw[map] (K) -- (G1);
\draw[map] (G1) -- node[lab, above, pos=.52] {$\rho$} (H1);

\draw[map] (G2) -- node[lab, above, pos=.62, yshift=3pt] {$\rho$} (H2);

\draw[map] (C1) -- node[lab, above, pos=.52] {$\pi$} (P1);
\draw[map] (C2) -- node[lab, above, pos=.52] {$\pi$} (P2);

\draw[vmap] (G1) -- node[vlab, left, pos=.50, xshift=-3pt] {$\Pi$} (C1);
\draw[vmap] (H1) -- node[vlab, right, pos=.50, xshift=3pt] {$\alpha$} (P1);

\draw[vmap] (G2) -- node[vlab, left, pos=.50, xshift=-3pt] {$\Pi$} (C2);
\draw[vmap] (H2) -- node[vlab, right, pos=.50, xshift=3pt] {$\alpha$} (P2);

\draw[map] (G1) -- node[lab, above left] {$g$} (G2);
\draw[map] (H1) -- node[lab, above left] {$\bar g$} (H2);

\draw[map] (C1) -- node[lab, above left] {$f$} (C2);
\draw[map] (P1) -- node[lab, above left] {$p$} (P2);

\end{tikzpicture}
\end{figure}

\begin{defi}\label{def:skew}\normalfont
Let $G$ be an algebraic group over $\Cbb$, let
$\rho:G\twoheadrightarrow H$ be a surjective algebraic group homomorphism
to a $1$-dimensional algebraic group $H$.
\begin{enumerate}
\item A dominant rational map $\Pi:G\dashrightarrow \Cbb^2$ is said to be
in \emph{skew-product form} (with respect to $\rho$) if there exist a
nonconstant rational map $\alpha:H\dashrightarrow \Pbb^1$ and a rational
map $\beta:G\dashrightarrow \Pbb^1$ such that
\begin{align*}
\Pi=(\alpha\circ\rho,\beta).
\end{align*}
\item An \emph{affine self-map} of $G$ is a surjective map of the form
\begin{align*}
g=t\circ g_0,
\end{align*}
where $g_0:G\to G$ is an algebraic group endomorphism and $t\in G$.
We say that $g$ is in \emph{skew-product form} (with respect to $\rho$)
if there exists an affine self-map $\bar g:H\to H$ such that
\begin{align*}
\rho\circ g = \bar g\circ \rho.
\end{align*}
This means that the top face of the cubic diagram is commutative (and $g$ sends fibers of $\rho$ to fibers of $\rho$).
\end{enumerate}
\end{defi}

\begin{defi}\label{def: semiconjugate}\normalfont
Let $f$ be a skew product of $\Cbb^2$ of degree $d\ge 2$.
We say that $f$ is \emph{semiconjugate} to an affine self-map $g$ in skew product form of a $2$-dimensional algebraic group $G$ if there exists a dominant generically finite rational map 
$\Pi:G\dashrightarrow\Cbb^2$ in skew product form such that
\begin{align}\label{eq: semiconjug}
    f\circ\Pi = \Pi\circ g.
\end{align}
\end{defi}
This means that the (left face of) cubic diagram is commutative. Here $\pi$ is the first-coordinate projection.
\subsection{Main results}

\begin{thm}\label{thm: main}
    Let $f$ be a regular polynomial skew product of $\Cbb^2$ of degree $d\ge 2$. Then the following assertions are equivalent.
    \begin{enumerate}
        \item $f$ is special.
        \item $f$ is semiconjugate to an affine self-map $g$ in skew product form of a 2-dimensional algebraic group $G$.
        \item All multipliers of $f$ are contained in a fixed number field $K$.
    \end{enumerate}
\end{thm}

The key point to prove Theorem~\ref{thm: main} is Theorem~\ref{thm: main'} which states that if $p$ is special and if there exist infinitely many periodic points such that the first return maps $Q^n_{z_0}$ is special, then $f$ is itself a special skew product. To do so, we prove a decomposition lemma \emph{\`a la} Ritt (Lemma~\ref{lem: decomposition}), which reduces the determination of $q$ from its non-autonomous iteration $Q^n_{z_0}$ to the functional equation~\eqref{eq:heart}. Finally, we can show that if we suppose (2) or (3) in Theorem~\ref{thm: main}, then we can reduce to the case of Theorem~\ref{thm: main'} by applying respectively Lemma~\ref{lem:onedim-affine-semi} and Theorem~\ref{thm: Huguin}.

Finally, note that for invertible polynomial automorphisms of $\Cbb^2$, Cantat and Dujardin considered the question of the field of definitions of multipliers for H\'enon maps in~\cite[Theorem D]{cantat2024rigidityresultspolynomialautomorphisms}.

\section*{Acknowledgments} The author would like to thank Thomas Gauthier for helpful discussions.

\section{Dickson polynomial and decomposition lemma}
Our main result in this section is the decomposition Lemma~\ref{lem: decomposition}. 
\subsection{Dickson polynomial}
In this subsection we introduce Dickson polynomials and present some properties. They have been studied mostly over finite fields. 
One can find most of the proofs in this subsection in~\cite{bookdickson} over an arbitrary ring. For the convenience of the reader, we reproduce them here.

By the fundamental theorem of symmetric polynomials, for each integer $d\ge 0$
there exists a unique polynomial $D_d(X,A)\in \Zbb[X,A]$ such that
\begin{align*}
u^d+v^d=D_d(u+v,uv).
\end{align*}
Substituting $u=t$ and $v=a/t$ yields
\begin{align*}
D_d\!\left(t+\frac{a}{t},a\right)=t^d+\frac{a^d}{t^d}.
\end{align*}

\begin{defi}\normalfont
The polynomial $D_d$ is called the \emph{Dickson polynomial} of degree $d$.
\end{defi}

If we set $s_d:=u^d+v^d$, then
\begin{align*}
s_0=2,\qquad s_1=u+v, \qquad \text{and}\qquad s_{d+1}=(u+v)s_d-uv\,s_{d-1}.
\end{align*}
Therefore the Dickson polynomials satisfy
\begin{align*}
D_0(x,a)=2,\qquad D_1(x,a)=x,\qquad D_2(x,a)=x^2-2a,
\end{align*}
and more generally
\begin{align}\label{eq: recursive-dick}
    D_{d+1}(x,a)=xD_d(x,a)-aD_{d-1}(x,a).
\end{align}

\begin{lem}\label{lem:basic-dick}
For all integers $m\ge 1$ and $d\ge 2$, the Dickson polynomials satisfy:
\begin{enumerate}
\item $D_d(x,0)=x^d$ and $D_d(x,1)=T_d(x),$
where $T_d$ is the Chebyshev polynomial of degree $d$.
\item There exists a polynomial $R_d(x,a)\in \Cbb[x,a]$ such that
\begin{align*}
D_d(x,a)=x^d-dax^{d-2}+R_d(x,a),
\qquad
\deg_x R_d\le d-4.
\end{align*}
\item $D_m(D_d(x,a),a^d)=D_{md}(x,a).$
\item For every $\lambda\in\Cbb^\ast$, $D_d(\lambda x,\lambda^2 a)=\lambda^d D_d(x,a).$
\end{enumerate}
\end{lem}

\begin{proof}
Write $x=u+v$ and $a=uv$. Then
\begin{align*}
D_d(x,a)=D_d(u+v,uv)=u^d+v^d.
\end{align*}

\smallskip

\noindent
\text{(1)} If $a=0$, then $uv=0$, hence $D_d(x,0)=u^d+v^d=(u+v)^d=x^d.$

If $a=1$, then $uv=1$, so $v=u^{-1}$, and thus $D_d(u+u^{-1},1)=u^d+u^{-d},$
which is the defining identity of the Chebyshev polynomial $T_d$.
Hence $D_d(x,1)=T_d(x).$

\smallskip
\noindent
\text{(2)} We prove by induction on $d$ that
\begin{align*}
D_d(x,a)=x^d-dax^{d-2}+R_d(x,a),
\qquad
\deg_x R_d\le d-4.
\end{align*}
For $d=2$ and $d=3$, we have $D_2(x,a)=x^2-2a$ and $D_3(x,a)=x^3-3ax$.
So the claim holds.

Assume now that $d\ge 3$ and that the claim holds for $d$ and $d-1$. Thus
\begin{align*}
D_d(x,a)&=x^d-dax^{d-2}+R_d(x,a),
\qquad \deg_x R_d\le d-4,\\
D_{d-1}(x,a)&=x^{d-1}-(d-1)ax^{d-3}+R_{d-1}(x,a),
\qquad \deg_x R_{d-1}\le d-5.
\end{align*}
Using the recurrence relation~\eqref{eq: recursive-dick}, we get
\begin{align*}
D_{d+1}(x,a)
&=xD_d(x,a)-aD_{d-1}(x,a)\\
&=x\bigl(x^d-dax^{d-2}+R_d(x,a)\bigr)
-a\bigl(x^{d-1}-(d-1)ax^{d-3}+R_{d-1}(x,a)\bigr)\\
&=x^{d+1}-(d+1)ax^{d-1}
+\Bigl(xR_d(x,a)+(d-1)a^2x^{d-3}-aR_{d-1}(x,a)\Bigr).
\end{align*}
The remainder term in parentheses has degree in $x$ at most $d-3$.
Hence
\begin{align*}
D_{d+1}(x,a)=x^{d+1}-(d+1)ax^{d-1}+R_{d+1}(x,a),
\qquad
\deg_x R_{d+1}\le d-3,
\end{align*}
which proves the claim.

\smallskip

\noindent
\text{(3)} Since $D_d(x,a)=u^d+v^d$ and $a^d=(uv)^d=u^dv^d$,
we have
\begin{align*}
D_m(D_d(x,a),a^d)&=D_m(u^d+v^d,u^dv^d)\\
&=u^{md}+v^{md}=D_{md}(x,a)
\end{align*}

\noindent
\text{(4)} We have $\lambda x=(\lambda u)+(\lambda v)$ and $\lambda^2 a=(\lambda u)(\lambda v),$
hence
\begin{align*}
D_d(\lambda x,\lambda^2 a)
&=D_d\bigl((\lambda u)+(\lambda v),(\lambda u)(\lambda v)\bigr)\\
&=(\lambda u)^d+(\lambda v)^d=\lambda^dD_d(x,a)
\end{align*}
\end{proof}

\begin{lem}~\label{lem: degm}
Let $\varphi\in \Cbb[z]$ be a polynomial of degree $m\ge 0$. Let $d\ge 2$. Assume that
\begin{align*}
\deg D_d(w,\varphi(z))=d
\end{align*}
as a polynomial in the two variables $w$ and $z$, with total degree. Then
\begin{align*}
\deg \varphi\in\{0,1,2\}.
\end{align*}
\end{lem}

\begin{proof}
By Lemma~\ref{lem:basic-dick}(2), we have
\begin{align*}
D_d(x,a)=x^d-dax^{d-2}+R_d(x,a),
\qquad
\deg_x R_d\le d-4.
\end{align*}
Substituting $x=w$ and $a=\varphi(z)$, we obtain
\begin{align*}
D_d(w,\varphi(z))
=
w^d-d\,\varphi(z)\,w^{d-2}+R_d(w,\varphi(z)),
\end{align*}
where every term of $R_d(w,\varphi(z))$ has degree in $w$ at most $d-4$.

Suppose for contradiction that $m\ge 3$. Then the term $-d\,\varphi(z)\,w^{d-2}$
contains a nonzero monomial of the form
$c\,z^m w^{d-2},c\neq 0,$
whose total degree is $m+d-2>d.$
Moreover, this monomial cannot be canceled by any other term, because every
other term has degree in $w$ at most $d-4$. Hence
\begin{align*}
\deg D_d(w,\varphi(z))\ge m+d-2>d,
\end{align*}
contradicting the assumption.
\end{proof}

\begin{lem}\label{lem:dickson-conjugacy}
Let $d\ge 2$ and let $\zeta\in \Cbb$. Then $D_d(x,\zeta)$ is affinely conjugate to $\pm T_d$ if and only if $\zeta^{d-1}=1$.
\end{lem}

\begin{proof}
Suppose $\zeta^{d-1}=1$.
Choose $\lambda\in\Cbb^*$ such that $\lambda^2 =\zeta$
and set $L(x)=\lambda x.$
By Lemma~\ref{lem:basic-dick}(1) and (4), we have
\begin{align*}
D_d(\lambda x,\zeta)=D_d(\lambda x,\lambda^2)=\lambda^dD_d(x,1)=\lambda^dT_d(x).
\end{align*}
Therefore $L^{-1}\circ D_d(\,\cdot\,,\zeta)\circ L(x)
=\lambda^{-1}D_d(\lambda x,\zeta)
=\lambda^{d-1}T_d(x).$
Since $(\lambda^{d-1})^2=\zeta^{d-1}=1$
we get $\lambda^{d-1}\in\{\pm1\}.$
Hence $D_d(x,\zeta)$ is affinely conjugate to $\pm T_d$.

Conversely, assume that $D_d(x,\zeta)$ is affinely conjugate to $\pm T_d$. Then there exist $\sigma\in\{\pm1\}$ and an affine map $L(x)=ux+v$
such that $D_d(x,\zeta)=L^{-1}\circ (\sigma T_d)\circ L(x).$
Equivalently,
\begin{align}\label{eq: EQ000}
    D_d(x,\zeta)=\frac{\sigma T_d(ux+v)-v}{u}.
\end{align}

Since both $D_d(x,\zeta)$ and $T_d(x)$ are monic of degree $d$, comparing the leading coefficients gives $\sigma u^{d-1}=1.$
Moreover, both $D_d(x,\zeta)$ and $T_d(x)$ have vanishing $x^{d-1}$-coefficient. The coefficient of $x^{d-1}$ in the right-hand side of ~\eqref{eq: EQ000} is $\sigma d u^{d-1}v/u = dv/u$,
hence $v=0$.

Therefore
\begin{align*}
D_d(x,\zeta)=\frac{\sigma T_d(ux)}{u}.
\end{align*}
Using Lemma~\ref{lem:basic-dick}(2), we have $D_d(x,\zeta)=x^d-d\zeta x^{d-2}+\cdots$ and
\begin{align*}
\frac{\sigma T_d(ux)}{u}
=
\sigma u^{d-1}x^d-d\sigma u^{d-3}x^{d-2}+\cdots
=
x^d-d\sigma u^{d-3}x^{d-2}+\cdots,
\end{align*}
so comparing the coefficient of $x^{d-2}$ yields $\zeta=\sigma u^{d-3}.$
Since $\sigma u^{d-1}=1$, we get $\zeta=u^{-2}.$
Hence
\begin{align*}
\zeta^{d-1}=u^{-2(d-1)}=(\sigma u^{d-1})^{-2}=1.
\end{align*}
This proves the converse.
\end{proof}

\subsection{Decomposition lemma}
In this subsection we prove Lemma~\ref{lem: decomposition}.
The following is a variant of Ritt's decomposition theorems proved by Zieve and M\"uller~\cite[Corollary 2.9]{arXiv:0807.3578}
\begin{lem}
    Suppose $a,b,c,d$ are non-constant polynomials such that $a\circ b=c\circ d$ and $\deg (a)=\deg(c)$, then there exists an affine map $A$ such that $a=c\circ A$ and $b=A^{-1}\circ d$.
\end{lem}
By induction, we have the following.
\begin{cor}\label{cor: ritt}
Let $n\geq 1$ be an integer. Let $F_j, G_j, 0\le j \le n-1$ be non-constant polynomials such that
\begin{align*}
F_{n-1}\circ\cdots\circ F_0
 =G_{n-1}\circ\cdots\circ G_0
\end{align*}
and $\deg F_j=\deg G_j$, for all $0\le j\le n-1.$
Then there exist affine maps $A_1,\dots,A_{n-1}$ such that
\begin{align*}
F_0&=A_1^{-1}\circ G_0,\\
F_j&=A_{j+1}^{-1}\circ G_j\circ A_{j}\qquad (1\le j\le n-2),\\
F_{n-1}&=G_{n-1}\circ A_{n-1}.
\end{align*}
\end{cor}
\begin{proof}
We argue by induction on $n$. The case $n=1$ is trivial.

Assume $n\ge 2$, and set
\begin{align*}
U:=F_{n-1}\circ\cdots\circ F_1,
\qquad
V:=G_{n-1}\circ\cdots\circ G_1.
\end{align*}
Then $U\circ F_0=V\circ G_0.$
Since $\deg F_0=\deg G_0$, we have $\deg U=\deg V$. By the lemma, there
exists an affine map $A_1$ such that $F_0=A_1^{-1}\circ G_0$ and $U=V\circ A_1.$
Hence
\begin{align*}
F_{n-1}\circ\cdots\circ F_1
=
G_{n-1}\circ\cdots\circ G_2\circ (G_1\circ A_1).
\end{align*}
Applying the induction hypothesis to these two decompositions of length $n-1$,
we obtain affine maps $A_2,\dots,A_{n-1}$ such that
\begin{align*}
F_1&=A_2^{-1}\circ G_1\circ A_1,\\
F_j&=A_{j+1}^{-1}\circ G_j\circ A_j \qquad (2\le j\le n-2),\\
F_{n-1}&=G_{n-1}\circ A_{n-1}.
\end{align*}
Together with $F_0=A_1^{-1}\circ G_0$, this proves the result.
\end{proof}

\begin{lem}\label{lem: decomposition}
Let $f_i, 0\le i\le n-1$ be monic polynomials of degree $d\geq 2$. The composition $f_{n-1}\cdots f_0$ is conjugate to $z^{d^n}$ or $\pm T_{d^n}$ if and only if
there exists complex numbers $c_j, \ell_j, 0\le j \le n-1$ such that
\begin{align*}
    f_j(x)=D_d(x-c_j,\ell_j)+c_{j+1},\qquad \ell_{j+1}=\ell_j^d,
\end{align*}
where the index is modulo $n$, and in the power map case $\ell_j=0$; and in the Chebyshev case, $\ell_0^{d^n-1}=1$.
\end{lem}

\begin{proof}
\textbf{Step 1. We first prove the direct implication.}

\noindent
\textbf{Power case.}
Applying Corollary~\ref{cor: ritt} with $F_j=f_j$ and $G_j(z)=z^d$, there exist affine maps $A_0,\cdots,A_n=A_0$ such that
\begin{align*}
f_j(x)=A_{j+1}^{-1}\bigl(A_j(x)^d\bigr).
\end{align*}
Write $A_j(x)=\alpha_j(x-c_j),$
then
\begin{align*}
f_j(x)=\frac{\alpha_j^d}{\alpha_{j+1}}(x-c_j)^d+c_{j+1}=(x-c_j)^d+c_{j+1},
\end{align*}
where the second equality is because $f_j$ is monic. By Lemma~\ref{lem:basic-dick}(1), $f_j(x)=D_d(x-c_j,0)+c_{j+1}$. So we take $\ell_j=0$.

\smallskip

\noindent
\textbf{Chebyshev case.}
Here
\begin{align*}
f_j(x)=A_{j+1}^{-1}\bigl(\sigma_jT_d(A_j(x))\bigr)
=\frac{\sigma_j}{\alpha_{j+1}}\,T_d(\alpha_j(x-c_j))+c_{j+1},
\end{align*}
where $\sigma_j\in\{\pm1\}$.
Since $f_j$ is monic and $T_d$ is monic, comparing leading coefficients gives $\alpha_{j+1}=\sigma_j\,\alpha_j^d.$
Define $\ell_j:=\alpha_j^{-2}.$
Then
\begin{align*}
\ell_{j+1}
=
\alpha_{j+1}^{-2}
=
(\sigma_j\alpha_j^d)^{-2}
=
\alpha_j^{-2d}
=
\ell_j^d.
\end{align*}
Since $\alpha_0=\alpha_n$, we have $\ell_0=\ell_n=\ell^{d^n}_0$, and $\ell^{d^n-1}=1$.

By Lemma~\ref{lem:basic-dick}(1) and (4), we have $\alpha_j^{-d}T_d(\alpha_j y)=D_d(y,\alpha_j^{-2})=D_d(y,a_j).$
Therefore
\begin{align*}
f_j(x)
&=
\frac{\sigma_j}{\alpha_{j+1}}\,T_d(\alpha_j(x-c_j))+c_{j+1}\\
&=
\alpha_j^{-d}T_d(\alpha_j(x-c_j))+c_{j+1}
=
D_d(x-c_j,\ell_j)+c_{j+1}.    
\end{align*}

\textbf{Step 2. We now prove the converse.}
Let $\tau_c(x):=x+c$. Then
\begin{align*}
f_j=\tau_{c_{j+1}}\circ D_d(\,\cdot\,,\ell_j)\circ \tau_{-c_j},
\end{align*}
hence
\begin{align*}
f_{n-1}\circ\cdots\circ f_0
=
\tau_{c_0}\circ
D_d(\,\cdot\,,\ell_{n-1})\circ\cdots\circ D_d(\,\cdot\,,\ell_0)
\circ \tau_{-c_0}.
\end{align*}
Since $\ell_{j+1}=\ell_j^d$, we have inductively $\ell_j=\ell_0^{d^j}.$
Therefore, by repeated use of Lemma~\ref{lem:basic-dick}(3),
we obtain
\begin{align*}
D_d(\,\cdot\,,\ell_{n-1})\circ\cdots\circ D_d(\,\cdot\,,\ell_0)
=
D_{d^n}(\,\cdot\,,\ell_0).
\end{align*}
Thus
\begin{align*}
f_{n-1}\circ\cdots\circ f_0
=
\tau_{c_0}\circ D_{d^n}(\,\cdot\,,\ell_0)\circ \tau_{-c_0}.
\end{align*}

If $\ell_0=0$, then $D_{d^n}(x,0)=x^{d^n}$, so the composition is affinely
conjugate to $z^{d^n}$.

Assume now that $\ell_0^{d^n-1}=1$. Set $N:=d^n$, and
choose $\mu\in\Cbb^*$ such that $\mu^2=\ell_0$. Using the scaling identity Lemma~\ref{lem:basic-dick}(4),
we get
\begin{align*}
\mu^{-1}D_N(\mu x,\ell_0)=\mu^{N-1}T_N(x).
\end{align*}
Since
\begin{align*}
(\mu^{N-1})^2=\mu^{2(N-1)}=\ell_0^{N-1}=1,
\end{align*}
we have $\mu^{N-1}=\pm1$. Hence $D_N(x,\ell_0)$ is affinely conjugate to
$\pm T_N$, and therefore so is $f_{n-1}\circ\cdots\circ f_0$.
\end{proof}

\section{Proof of Theorem~\ref{thm: main}}\label{sect: proof}
In this section we prove our main result. Given a polynomial $H(z)$, the notation $[z^k]H$ means the coefficient of $z^k$ in $H(z)$.
\subsection{Special on the base and on infinitely many fibers implies special}
We first prove the following key theorem.

\begin{thm}\label{thm: main'}
   Let $f(z,w)=\big(p(z),q(z,w) \big)$ be a skew product of degree $d\geq 2$ such that $p$ is special. Suppose moreover there exist infinitely many $p$-periodic points $z_0$ of exact period $n=n(z_0)$ such that $Q_{z_0}^n$ is special, then $f$ is special.

   Conversely, if $f$ is special, then $p$ and $Q_{z_0}^n$ are special for all $p$-periodic points $z_0$ of exact period $n=n(z_0)$.
\end{thm}
\begin{proof}
\textbf{We first prove the first assertion.}
\textbf{Step 1.}
Up to a triangularly conjugation, we may assume that $p\in \Ecal_d$.

By Lemma~\ref{lem: decomposition} applied to the decomposition $Q_{z_0}^{(n)}$,
there exist complex numbers $c_i$, $\ell_j$ $0 \le i,j \le n-1$
such that
\begin{align}\label{eq: mainproof_1}
    q(z_j,w)=D_d(w-c_j,\ell_j)+c_{j+1} \text{  and  }\ell_{j+1}=\ell_j^d
\qquad (0\leq j\leq n-1)
\end{align}
where the index is modulo $n.$

Set
\begin{align*}
c(z):=-\frac{1}{d}[w^{d-1}]\,q(z,w).
\end{align*}
By Lemma~\ref{lem:basic-dick}(2), we have $c_j =-\frac{1}{d} [w^{d-1}]D_d(w-c_j,\ell_j),$
and thus $c(z_j)=c_j.$
Therefore, by Eq~\eqref{eq: mainproof_1}, $c(p(z_j))=c(z_{j+1})=c_{j+1}.$

Define the conjugation $\Psi(z,w)=(z,w+c(z))$, we get
\begin{align*}
    \widetilde f(z,u):=\Psi^{-1}\circ f\circ\Psi=(p(z),\widetilde q(z,u) ),
\end{align*}
where $\widetilde q(z,u):=q(z,u+c(z))-c(p(z))$ is a centered polynomial of degree $d$, i.e., $[u^{d-1}]\widetilde q=0$. Moreover 
\begin{align*}
\widetilde q(z_j,u)
=
q(z_j,u+c_j)-c_{j+1}
=
D_d(u,\ell_j).
\end{align*}

Define 
\begin{align*}
\varphi(z):=-\frac{1}{d}[u^{d-2}]\,\widetilde q(z,u) \qquad \text{and} \qquad R(z,u):=\widetilde q(z,u)- D_d(u,\varphi(z)).
\end{align*}
By lemma~\ref{lem:basic-dick}(2), $[u^{d-2}]D_d(u,\ell_j)=-d \ell_j$ and thus $\varphi(z_j)=\ell_j.$ Therefore, for infinitely many periodic points $z$, $R(z,u)\equiv0$ as a polynomial in $u.$
Since any infinite subset of $\mathbb{A}^1$ is Zariski dense, the coefficient of $R(z,u)$ as a polynomial in $u$ vanishes identically in $z$. In other words, we have 
\begin{align*}\widetilde q(z,u)\equiv D_d(u,\varphi(z)).\end{align*}
Moreover, since $\ell_j^d=\ell_{j+1}$, we have
\begin{align*}\varphi(z_j)^d=\ell_j^d=\ell_{j+1}=\varphi(z_{j+1})=\varphi(p(z_j)).
\end{align*}
Again by Zariski density of periodic points, we thus obtain the functional equation
\begin{align}\tag{$\heartsuit$}\label{eq:heart}
        \varphi(p(z))=\varphi(z)^d.
    \end{align}

\textbf{Step 2. We now solve this functional equation.}

\textbf{Power case.}
Assume $p(z)=z^d$. Then Eq.~\eqref{eq:heart} is $\varphi(z^d)=\varphi(z)^d$.
If $\varphi\equiv 0$, then $\widetilde f(z,w)=(z^d,w^d)$. Assume $\varphi\not\equiv 0$.

If $\alpha\neq 0$ is a zero of $\varphi$, then every $d^k$-th root of $\alpha$ is also a zero of $\varphi$, giving infinitely many distinct zeros, impossible for a polynomial. Hence the only possible zero of $\varphi$ is $0$.

Therefore $\varphi(z)=\zeta z^m$
for some $\zeta\in \mathbb{C}^\ast$ and some $m\geq 0$. By Lemma~\ref{lem: degm}, $m=0,1,2$. 
Substituting into $\varphi(z^d)=\varphi(z)^d$
gives $\zeta z^{md}=\zeta^d z^{md},$
hence $\zeta^{d-1}=1.$ If $m=0$, then $\widetilde f$ is triangularly conjugate to $\big(z^d,\pm T_d(w)\big)$, else $\widetilde f(z,w)=\big(z^d,D_d(w,\zeta z^m) \big)$.

\textbf{Chebyshev case.}
Assume $p=\sigma T_d$, $\sigma\in\{\pm\}$.
Choose $\mu\in\Cbb^*$ such that $\mu^{d-1}=\sigma,$
and set $\nu:=\mu^2.$
Then $\nu^{d-1}=1.$
By Lemma~\ref{lem:basic-dick}(4),
\begin{align*}
D_d(\mu x,\mu^2)=\mu^d D_d(x,1)=\mu^d T_d(x),
\end{align*}
hence
\begin{align*}
\mu^{-1}D_d(\mu x,\nu)=\mu^{d-1}T_d(x)=\sigma T_d(x)=p(x).
\end{align*}
Thus, if $A(x):=\mu x$, then $p=A^{-1}\circ D_d(\,\cdot\,,\nu)\circ A.$

Set $g:=\varphi\circ A^{-1}.$
Then the identity $\varphi\circ p(z) =\varphi(z)^d$
becomes
\begin{align*}
g\circ D_d(\,\cdot\,,\nu)=g^d.
\end{align*}

Now define
\begin{align*}
\pi_\nu(t):=t+\frac{\nu}{t},
\qquad
u(t):=g(\pi_\nu(t))\in\Cbb[t,t^{-1}].
\end{align*}
Therefore
\begin{align*}
D_d\!\left(\pi_\nu(t),\nu\right)
=
D_d\!\left(t+\frac{\nu}{t},\nu\right)
=
t^d+\frac{\nu^d}{t^d}
=
t^d+\frac{\nu}{t^d}
=
\pi_\nu(t^d).
\end{align*}
Hence
\begin{align*}
u(t^d)=u(t)^d.
\end{align*}
Now we need to solve this functional equation.
Suppose first that $u$ has a zero in $\Cbb^*$. Then all iterated $d^n$-th
roots of that zero are again zeros of $u$, yielding infinitely many zeros in
$\Cbb^*$, impossible for a Laurent polynomial. Thus 
\begin{align*}
u(t)=\eta t^m
\end{align*}
for some $\eta\in\Cbb^*$ and $m\in \Zbb$.

On the other hand, $\pi_\nu(t)=\pi_\nu(\nu/t),$
so $u(t)=u(\nu/t).$
Thus $\eta t^m=\eta (\nu/t)^m=\eta\,\nu^m t^{-m}$,
which forces $m=0$. Hence $u$ is constant, so $g$ is constant, and
therefore $\varphi$ is constant.
Finally, if $\varphi\equiv\zeta$, then $\zeta=\zeta^d,$
so either $\zeta=0$ in which case $\widetilde f(z,w)=\big(\pm T_d(z), w^d\big)$ or else $\zeta^{d-1}=1$, in which case $\widetilde f(z,w)=\big(\pm T_d(z), \pm T_d(w)\big)$.\\

\medskip
\textbf{We now prove the converse.}
Assume $p\in \Ecal_d$ and
\begin{align*}
q(z,w)=D_d(w,\varphi(z))
\qquad\text{and}\qquad
\varphi\circ p(z)=\varphi(z)^d.
\end{align*}
We claim that for every $n\geq 1$, $Q_z^{n}(w)=D_{d^n}(w,\varphi(z)).$
We prove this by induction on $n$.
For $n=1$, this is exactly the definition of $q$.

Assume it holds for some $n\geq 1$. Then
\begin{align*}
Q_z^{n+1}(w)
&=
q_{p^n(z)}\bigl(Q_z^{n}(w)\bigr)\\
&=
D_d\!\Bigl(Q_z^{n}(w),\,\varphi(p^n(z))\Bigr)\\
&=
D_d\!\Bigl(D_{d^n}(w,\varphi(z)),\,\varphi(z)^{d^n}\Bigr)\\
&=
D_{d^{n+1}}(w,\,\varphi(z)).
\end{align*}
This proves the claim.

Let $z$ be a periodic point of period $n$ for $p$. If $\varphi(z)=0$, then $Q_z^{n}(w)=w^{d^n}$.
If $\varphi(z)\neq 0$, then from $\varphi(p^n(z))=\varphi(z)^{d^n}=\varphi(z)$,
we get $\varphi(z)^{d^n-1}=1.$
Therefore Lemma~\ref{lem:dickson-conjugacy} shows that $D_{d^n}(x,\varphi(z))$
is affinely conjugate to $\pm T_{d^n}$.
\end{proof}

\subsection{Proof of (1)$\iff$(2)}
We need the following lemma, see e.g., Silverman's book~\cite{SilvermanBook} and Milnor's survey~\cite{milnor06}.
\begin{lem}\label{lem:onedim-affine-semi}
Let $L$ be an algebraic group of dimension $1$, let $\varphi:L\to L$ be an affine map, and let $f:\Pbb^1\to\Pbb^1$ and
$\Pi:L\dashrightarrow\Pbb^1$ be nonconstant rational maps such that
\begin{align*}
f\circ \Pi = \Pi\circ \varphi.
\end{align*}
Then the following hold.
\begin{enumerate}
\item If $L=\Gbb_a$, then $\deg f = 1$.
\item If $L=\Gbb_m$, then $f$ is conjugate to $z\mapsto z^{\pm \deg f}$ or $\pm T_{\deg f}$.
\item If $L$ is an elliptic curve, then $f$ is a Latt\`es map.
\end{enumerate}
\end{lem}

We write the group law multiplicatively.
We already remarked that (1) implies (2) in Introduction. In fact, $\pm$Chebyshev maps are semiconjugate to power maps and we explained that $f_{2,m}(z,w)$ is conjugate to $(z^d,\pm T_d(w))$.
So it suffices to prove the converse.

Write $\Pi=(\alpha\circ\rho,\beta)$ and $K=\ker(\rho)$.
Since $\Pi$ is dominant and generically finite, there exists a finite set
$E\subset H$ such that for every $u\in H\setminus E$, the specialization
\begin{align*}
\beta_u:=\beta|_{\rho^{-1}(u)}: \rho^{-1}(u)\dashrightarrow \Pbb^1
\end{align*}
is a nonconstant rational map, and its degree is independent of $u$.

The first-coordinate relation in Eq.~\eqref{eq: semiconjug} reads $p\circ \alpha = \alpha\circ \bar g.$
Apply Lemma~\ref{lem:onedim-affine-semi} to $(L,\varphi,f,\Pi)=(H,\bar g,p,\alpha).$
Since $p$ is a polynomial of degree $d\ge 2$, we get $H\neq \Gbb_a$ from
Lemma~\ref{lem:onedim-affine-semi}(1), and $H$ cannot be elliptic by
Lemma~\ref{lem:onedim-affine-semi}(3). Therefore $H\simeq \Gbb_m$,
and $p$ is special.

Choose a periodic point $u_0\in H\setminus E$ for $\bar g$. Let $N$ be the exact period of $u_0$, and put $z_0:=\alpha(u_0).$
Then $z_0$ is periodic for $p$, with period $n(z_0) $ dividing $N$.

Pick $\xi_0\in \rho^{-1}(u_0)$ and let
\begin{align*}
\tau_{\xi_0}:K\longrightarrow \rho^{-1}(u_0),\qquad k\longmapsto \xi_0 k.
\end{align*}
Since $\rho\circ g^N = \bar g^N\circ \rho$ and $\bar g^N(u_0)=u_0$, the map
$g^N$ preserves the fiber $\rho^{-1}(u_0)$. Therefore $\phi_{u_0}:=\tau_{\xi_0}^{-1}\circ g^N\circ \tau_{\xi_0}$
is an affine self-map of $K$, and if we set $\widetilde\beta_{u_0}:=\beta_{u_0}\circ \tau_{\xi_0}$,
then the second coordinate of $f^N\circ\Pi=\Pi\circ g^N$ yields 
\begin{align*}
Q_{z_0}^N\circ \widetilde\beta_{u_0}
=
\widetilde\beta_{u_0}\circ \phi_{u_0}.
\end{align*}
Apply Lemma~\ref{lem:onedim-affine-semi} to $(L,\varphi,A,X)=(K,\phi_{u_0},Q_{z_0}^N,\widetilde\beta_{u_0}).$
Since $Q_{z_0}^N$ is a polynomial of degree $d^N>1$, we get $K\simeq \Gbb_m,$ and the polynomial $Q_{z_0}^N$ is special and thus so is $Q_{z_0}^{n(z_0)}$ since $n(z_0)$ divides $N$.
Now $G$ is an extension of $\Gbb_m$ by $\Gbb_m$, so we have $G\simeq \Gbb_m^2.$

The set of periodic points $u_0\in \Gbb_m\setminus E$ is infinite, and the
nonconstant rational map $\alpha:\Pbb^1\dashrightarrow\Pbb^1$ has finite
fibers on its domain of definition. Therefore the corresponding points
$z_0=\alpha(u_0)$ form an infinite set of periodic points of $p$.
For each such $z_0$, the first return map $Q_{z_0}^{n(z_0)}$ is special. Therefore $f$ is special by Theorem~\ref{thm: main'}.

\subsection{Proof of (1)$\iff$(3)}
Fix a periodic orbit
\begin{align*}
z_0,\dots,z_{n-1},
\qquad z_{j+1}=p(z_j),
\qquad z_n=z_0.
\end{align*}
Then the multiplier of $z_0$ is $(p^n)'(z_0)$ and $\partial_w(Q^n_{z_0}(w))$. Since they are all contained in a fixed number field, Theorem~\ref{thm: Huguin} implies that, $p$ and $Q^n_{z_0}$ are special. Then by Theorem~\ref{thm: main'}, we infer that $f$ is special.

Conversely, if $f$ is special, then Theorem~\ref{thm: main'} implies that $p$ is special and  $Q^n_{z_0}$ is also for all periodic points $z_0$ of exact period $n=n(z_0)$. But special polynomials have all their multipliers contained in $\Qbb$. This finishes the proof.

\bibliographystyle{alpha}
\bibliography{mybio}
\end{document}